\documentclass{article}

\usepackage{etoolbox} 
\usepackage[shortlabels,inline]{enumitem}

\usepackage[immediate]{silence}
\usepackage{environ}
\usepackage{xparse}
\usepackage{listofitems}

\usepackage[
	a4paper,
	margin=24.5mm
]{geometry}
\usepackage[onehalfspacing]{setspace}

\usepackage{sectsty} 
\DeactivateWarningFilters[temp] 

\usepackage{amsmath, amssymb, amsthm}
\usepackage{mathtools}
\usepackage[a]{esvect}

\usepackage{array, makecell}

\usepackage{graphicx}
\usepackage[dvipsnames]{xcolor}
\usepackage{tikz}
\usepackage[pstarrows]{pict2e}

\usepackage[
    setpagesize=false,
    colorlinks=true,
    linkcolor=blue,
    citecolor=blue,
    pdfencoding=auto,
    psdextra,
]{hyperref}
\usepackage{cleveref}
\usepackage[
    sorting=nyt,
    maxbibnames=99
]{biblatex}

\usepackage[
    deletedmarkup=sout,
    commentmarkup=uwave,
    authormarkuptext=name
]{changes}

\DeclareFieldFormat{url}{} 
\DeclareFieldFormat[misc]{url}{\textsc{url}: \url{#1}}

\setlist[enumerate,1]{label=(\arabic*), ref=(\arabic*)}
\setlist[enumerate,3]{label=(\roman*), ref=(\roman*)}

\allsectionsfont{\boldmath} 

\makeatletter
\newcommand{\labeltext}[2]{%
    \@bsphack%
    \csname phantomsection\endcsname
    \def\tst{#1}%
    \def\refmarkup{}%
    \def\@currentlabel{\refmarkup{#1}}{\label{#2}}%
    \@esphack%
}
\makeatother

\newcounter{propcounter}
\newenvironment{proplist}{%
    \stepcounter{propcounter}%
    \begin{enumerate}[label = {\bfseries \Alph{propcounter}\arabic{enumi}}]%
}
{\end{enumerate}}

\crefname{subsection}{subsection}{subsections}

\newtheorem{theorem}{Theorem}[section]
\newtheorem{lemma}[theorem]{Lemma}

\newtheorem{proposition}[theorem]{Proposition}

\theoremstyle{definition}
\newtheorem{definition}[theorem]{Definition}

\newtheorem{question}[theorem]{Question}

\theoremstyle{plain}
\newtheorem{claim}{Claim}[theorem]
\newtheorem*{claim*}{Claim}

\makeatletter
\newenvironment{claimproof}[1][Proof]{\par
	\pushQED{\qed}%
	
	\normalfont \topsep6\p@\@plus6\p@\relax
	\trivlist
	\item[\hskip\labelsep
	\textit{#1}\@addpunct{.}~]\ignorespaces
}{%
	\popQED\endtrivlist\@endpefalse
}
\makeatother

\newlist{Cases}{enumerate}{3}
\setlist[Cases]{parsep=0pt plus 1pt}
\setlist[Cases,1]{wide=0pt, listparindent=\parindent,
    label = \textbf{Case~\arabic*:}, ref = \arabic*}
\setlist[Cases,2]{wide, labelindent=\parindent,
    label = \textbf{Case~\arabic{Casesi}}-(\arabic{Casesii}):}
\setlist[Cases,3]{wide, labelindent=1.5\parindent, topsep=5pt,
    label = \textbf{Case~\arabic{Casesi}}-(\arabic{Casesii})-(\arabic{Casesiii}):}

\crefname{Casesi}{case}{cases}
\crefname{conjecture}{conjecture}{conjectures}

\newcounter{statement}
\makeatletter
\let\c@statement\c@equation
\makeatother

\NewEnviron{statement}
	{\vspace{1ex}\par\noindent
	\refstepcounter{statement}%
	\hspace{1.5cm}%
	\parbox{\dimexpr\textwidth-3cm}{\BODY}%
	~\hspace{-1.55cm}\llap{(\arabic{statement})}%
	\vspace{1ex}\par}

\Crefformat{statement}{(#2#1#3)}

\pgfset{
    foreach/parallel foreach/.style args={#1in#2via#3}{
        evaluate=#3 as #1 using {{#2}[#3-1]}
    },
}

\makeatletter
\NewDocumentCommand{\statementapp}{m m m m}{%
	\def\rowBname{#1}%
	\def\rowAname{#2}%
	\def\listB{#3}%
	\def\listA{#4}%
	\newcount\listlen
	\let\arraycontent\empty
	\xappto\arraycontent{\text{\rowAname}}
	\foreach \Ae [count=\n] in \listA {%
		  \xappto\arraycontent{& \Ae}%
		  \global\let\listlen\n
	}%
        \gappto\arraycontent{\\ \hline}
        \xappto\arraycontent{\text{\rowBname}}
	\foreach \Be in \listB {%
		  \xappto\arraycontent{& \Be}%
	}%
	\[%
	\begin{array}{|r<{\hspace{0.3pc}}||*{\listlen}{c|}}
		\hline
		\arraycontent \\
		\hline
	\end{array}\]%
}

\makeatother

\robustify{\mathrm}

\newcommand{\TT}{\mathbf{T}}

\newcommand{\calA}{\mathcal{A}}

\newcommand{\calD}{\mathcal{D}}

\newcommand{\defeq}{\coloneqq}

\let\originalleft\left
\let\originalright\right
\renewcommand{\left}{\mathopen{}\mathclose\bgroup\originalleft}
\renewcommand{\right}{\aftergroup\egroup\originalright}

\makeatletter
\newcases{lrcases*}
  {\quad}
  {$\m@th{##}$\hfil}
  {{##}\hfil}
  {\lbrace}
  {\rbrace}
\makeatother

\newcommand{\Hpart}{\mathbf{H}}

\makeatletter
\DeclareRobustCommand{\vvrev}[1]{%
  \mathpalette\do@vv@rev{#1}%
}
\newcommand{\do@vv@rev}[2]{%
  \fix@vv@rev{#1}{+}%
  \reflectbox{$\m@th#1\vv{\reflectbox{$\fix@vv@rev{#1}{-}\m@th#1#2\fix@vv@rev{#1}{+}$}}$}%
  \fix@vv@rev{#1}{-}%
}
\newcommand{\fix@vv@rev}[2]{%
  \ifx#1\displaystyle
    \mkern#23mu
  \else
    \ifx#1\textstyle
      \mkern#23mu
    \else
      \ifx#1\scriptstyle
        \mkern#22mu
      \else
        \mkern#22mu
      \fi
    \fi
  \fi
}

\DeclareRobustCommand{\arc}{\vv}
\DeclareRobustCommand{\arcrev}{\vvrev}

\newcommand{\rarc}{\rightarrow}

\newcommand{\outdir}{\Rightarrow}
\newcommand{\indir}{\Leftarrow}

\newcommand{\crar}[1]{\xrightarrow{#1}}

\usetikzlibrary{scopes}
\usetikzlibrary{matrix,arrows,decorations.pathmorphing}
\usetikzlibrary{positioning}
\usetikzlibrary{decorations.markings}
\usetikzlibrary{calc}

\usetikzlibrary{decorations.pathreplacing,calligraphy,backgrounds}

\tikzset{
    vertex/.style={circle, fill=black, inner sep=0pt, minimum size=2mm},
    every label/.append style={rectangle, outer sep=3pt},
}

\tikzset{
    ->-/.style={
        decoration={markings,mark=at position #1 with {\arrow{>}}},
        postaction={decorate}},
    ->-/.default={0.75},
    -<-/.style={
        decoration={markings,mark=at position #1 with {\arrow{<}}},
        postaction={decorate}},
    -<-/.default={0.25},
}

\tikzset{
    highlight edge/.style = {preaction={draw,lightgray, double=lightgray,double distance=5pt}}
}

\definechangesauthor[name=Jaehyeon, color=Emerald!60!black]{JS}
\definechangesauthor[name=Jaehoon, color=green!50!black]{JK}
\definechangesauthor[name=Deb, color=blue]{D}
\definechangesauthor[name=Hyunwoo, color=orange]{H}

\title{Hamilton transversals in tournaments}
\author{
Debsoumya Chakraborti
    \thanks{Mathematics Institute, University of Warwick, Coventry, UK.
    E-mail: \texttt{debsoumya.chakraborti@warwick.ac.uk}. 
    Supported by the Institute for Basic Science (IBS-R029-C1), and the European Research Council (ERC) under the European Union Horizon 2020 research and innovation programme (grant agreement No.~947978).}
\and Jaehoon Kim
    \thanks{Department of Mathematical Sciences, KAIST, South Korea. Email: \texttt{jaehoon.kim@kaist.ac.kr, hyunwoo.lee@kaist.ac.kr}.
    Supported by the National Research Foundation of Korea (NRF) grant funded by the Korea government(MSIT) No.~RS-2023-00210430.}
\and Hyunwoo Lee\footnotemark[2]
    \thanks{Extremal Combinatorics and Probability Group
    (ECOPRO), Institute for Basic Science (IBS), South Korea.
    Partially supported by the Institute for Basic Science (IBS-R029-C4).}
\and Jaehyeon Seo
    \thanks{Department of Mathematics, Yonsei University, South Korea. E-mail: \texttt{jaehyeonseo@yonsei.ac.kr}.
    Supported by the National Research Foundation of Korea (NRF) grant funded by the Korea government(MSIT) No.~2022R1C1C1010300.}
}
\date{\today}

\begin{document}
\maketitle

\begin{abstract}
    It is well-known that every tournament contains a Hamilton path, and every strongly connected tournament contains a Hamilton cycle. 
    This paper establishes \textit{transversal} generalizations of these classical results. For a collection \(\TT=(T_1,\dots,T_m)\) of not-necessarily distinct tournaments on a common vertex set $V$, an $m$-edge directed graph $\mathcal{D}$ with vertices in $V$ is called a $\TT$-transversal if there exists a bijection \(\phi\colon E(\mathcal{D})\to [m]\) such that \(e\in E(T_{\phi(e)})\) for all \(e\in E(\mathcal{D})\). We prove that for sufficiently large $m$ with $m=|V|-1$, there exists a $\TT$-transversal Hamilton path. Moreover, if $m=|V|$ and at least $m-1$ of the tournaments $T_1,\ldots,T_m$ are assumed to be strongly connected, then there is a $\TT$-transversal Hamilton cycle. 
    In our proof, we utilize a novel way of partitioning tournaments which we dub \emph{\(\Hpart\)-partition}. 
\end{abstract}

\section{Introduction}\label{sec:intro}

Given a collection $\mathcal{F}=\{F_1,\ldots,F_m\}$ of sets, a set $X$ of size $m$ is called an \textit{$\mathcal{F}$-transversal} if there is a labelling $x_1,\ldots,x_m$ of the elements in $X$ such that $x_i\in F_i$ for each $i\in [m]$.
Transversals over various mathematical objects have been studied in the literature throughout the last few decades.
To name a few, such variants are extensively studied for Carath\'eodory's theorem~\cite{barany1982generalization,kalai2009colorful}, Helly's theorem~\cite{kalai2005topological}, the Erd\H{o}s-Ko-Rado theorem~\cite{aharoni2017rainbow}, Rota's basis conjecture~\cite{huang1994relations,pokrovskiy2020rota}, etc.

The same notion for graphs, i.e., transversals over a collection of graphs, is implicitly used in much of the literature and explicitly defined in~\cite{joos2020rainbow}. The same definition can be extended for related objects such as hypergraphs and directed graphs as follows. 

\begin{definition}
For a given collection $\mathcal{F}=(\calD_1,\dots, \calD_m)$ of graphs/hypergraphs/directed graphs with the same vertex set $V$, an $m$-edge graph/hypergraph/directed graph $\calD$ on the vertex set $V$ is an $\mathcal{F}$-transversal if there exists a bijection $\phi:E(\calD)\rightarrow [m]$ such that $e\in E(\calD_{\phi(e)})$ for all $e\in E(\calD)$.
\end{definition}

By interpreting each $\calD_i$ as the set of edges colored with the color $i$, the function $\phi$ is often called a \emph{coloring}.
We say the coloring \(\phi\) is \emph{rainbow} if it is injective, and we say \(\calD\) is a \textit{rainbow} if there is a rainbow coloring on it. 
Many classical results in extremal graph theory have been extended to such transversal settings, exhibiting interesting phenomena.

The classical Mantel's theorem states that any $n$-vertex graph with more than $\frac{1}{4}n^2$ edges must contain a triangle. Aharoni, DeVos, de la Maza, Montejano, and \v{S}\'{a}mal~\cite{aharoni2020rainbow} considered a transversal version of this showing that if a graph collection $\mathcal{G}=(G_1,G_2,G_3)$ on $n$ vertices satisfies ${\min_{i\in [3]} \{e(G_i)\} > \big(\frac{26-2\sqrt{7}}{81}\big)n^2}$, then it has a $\mathcal{G}$-transversal isomorphic to a triangle. Surprisingly, this condition with the irrational multiplicative constant is best possible. Furthermore, 
this captures an interesting phenomenon that $\frac{26-2\sqrt{7}}{81}$ is larger than $\frac{1}{4}$, which we obtain from Mantel's theorem. It is an interesting open problem to obtain a similar tight condition for the existence of a transversal of $K_r$ with $r>3$.  
A similar extremal problem, where instead of putting a condition on $\min_i \{e(G_i)\}$, to find the tight condition on $\sum_i e(G_i)$ for the existence of a $\mathcal{G}$-transversal isomorphic to a given graph is studied in \cite{chakraborti2024rainbow,keevash2004multicolour}.

Addressing a question of~\cite{aharoni2020rainbow}, Cheng, Wang, and Zhao~\cite{cheng2021pancyclicity} obtained an asymptotic version of the transversal generalization of the classical Dirac's theorem, and a complete resolution was independently obtained by Joos and the second author~\cite{joos2020rainbow}. They proved that if a graph collection $\mathcal{G}=(G_1,\ldots,G_n)$ on $n$ vertices satisfies ${\min_{i\in [n]} \{\delta(G_i)\} \ge n/2}$, then it has a $\mathcal{G}$-transversal isomorphic to a Hamilton cycle. Soon a number of results followed, finding similar tight conditions to find $\mathcal{G}$-transversals isomorphic to a few other graphs; see~\cite{gupta2023general,montgomery2022transversal}. Finally, the first two authors, Im, and Liu~\cite{chakraborti2023bandwidth} generalized these results by establishing the bandwidth theorem for graph transversals. 

Transversal generalizations were recently considered for hypergraphs and directed graphs; see,~\cite{cheng2023rainbow,cheng2021transversal}.
The above lines of research have a central theme which can be pinned by the following meta-question. 
\begin{question}\label{question}
    For a given graph/hypergraph/directed graph $\calD$ with $m$ edges, which properties $\mathcal{P}_n$ will ensure the following? Every collection $\mathcal{F}$ of $m$~graphs/hypergraphs/directed graphs on the same vertex set of size $n$ satisfying the property $\mathcal{P}_n$ contains a $\mathcal{F}$-transversal copy of $\calD$.
\end{question}

As all objects in the collection could be identical, a natural necessary criterion for such a property $\mathcal{P}_n$ is that it has to ensure that every $n$-vertex graph/hypergraph/directed graph with property $\mathcal{P}_n$ contains a copy of $\calD$. However, it is not necessarily sufficient as mentioned by the result of Aharoni, DeVos, de la Maza, Montejano, and \v{S}\'{a}mal~\cite{aharoni2020rainbow}. Thus, it is a natural problem to investigate when these properties directly carry over to the transversal generalizations from the original results. We answer this in positive for the transversal generalizations of the following two folklore results for sufficiently large tournaments. We call a tournament $T$ \textit{strongly connected} if for every pair of vertices $x,y$ in $T$, there is a directed path from $x$ to~$y$. Such a path or a cycle is \emph{Hamilton} if it contains all vertices of the given digraph.
In this paper, whenever we mention paths and cycles, we always refer to directed paths and directed cycles.
\begin{enumerate}[leftmargin=*]
    \item Every tournament contains a Hamilton path. \label{every tournament has a directed Hamilton path}
    \item Every strongly connected tournament contains a Hamilton cycle. \label{every strongly connected tournament contains a directed Hamilton cycle}
\end{enumerate}

In what follows, we always assume that \(\TT=(T_1,\dots,T_m)\) is a collection of tournaments on the common vertex set \(V(\TT)\). 
Our first main result is to establish a transversal version of~\ref{every tournament has a directed Hamilton path}. 

\begin{theorem}\label{thm:rainbow-dir-ham-path}
    For every sufficiently large \(n\), every collection $\TT$ of $n-1$ tournaments with $|V(\TT)| = n$ contains a transversal  Hamilton path.
\end{theorem}

We remark that the above result is not true when $n=3$ by considering two directed triangles with opposite orientation, i.e., the red and blue tournaments in \Cref{fig:(thm@rainbow-dir-ham-path)-rmk_description}.

\begin{figure}[htb]
    \centering
    \includegraphics{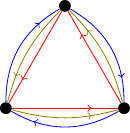}
    \caption{A counterexample for \Cref{thm:rainbow-dir-ham-cycle}, when \(n=3\).}
    \label{fig:(thm@rainbow-dir-ham-path)-rmk_description}
\end{figure}

Our next result establishes a transversal version of~\ref{every strongly connected tournament contains a directed Hamilton cycle}. In fact, we prove a slightly stronger statement.

\begin{theorem}\label{thm:rainbow-dir-ham-cycle}
    For every sufficiently large \(n\), every collection $\TT$ of $n$ tournaments with $|V(\TT)| = n$ satisfies the following. If all tournaments in $\TT$ possibly except one are strongly connected, then $\TT$ contains a transversal Hamilton cycle.
\end{theorem}

The above result is not true for $n=3$ by considering three directed triangles where two of them have the same orientation and the third one has the opposite orientation; see \Cref{fig:(thm@rainbow-dir-ham-path)-rmk_description}. The number `one' in \Cref{thm:rainbow-dir-ham-cycle} cannot be replaced by `two' in the above result as shown in the following proposition. 

\begin{proposition}\label{prop:tight-dir-ham-cycle}
    For every $n\ge 3$, there exists a collection $\TT$ of $n$ tournaments with $|V(\TT)| = n$, and all but two tournaments in $\TT$ are strongly connected such that $\TT$ does not contain a transversal Hamilton cycle.
\end{proposition}

\begin{figure}[htb]
    \centering
    \includegraphics[width=.7\textwidth]{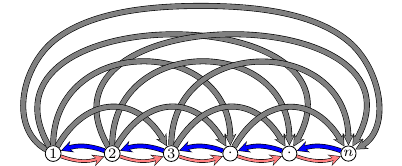}
    \caption{A construction for \Cref{prop:tight-dir-ham-cycle}: each red arrow consists of arcs in \(T_1\) and \(T_2\); each blue arrow consists of arcs in \(T_3,\dots, T_n\); and each gray arrow consists of arcs in \(T_1,\dots,T_n\).}
    \label{fig:(prop@tight-dir-ham-cycle)-pf_description}
\end{figure}

\begin{proof}[Proof of \Cref{prop:tight-dir-ham-cycle}]

Consider the transitive tournament $T$ on $[n]$ containing the arcs $\arc{ij}$ for every $1\le i< j\le n$, and the strongly connected tournament $T'$ on $[n]$ containing the arcs $\arc{ij}$ for every $2\le i+1< j\le n$ and the (backward) arcs $\arc{ji}$ for every $2\le i+1=j\le n$. Consider the collection $\TT=(T_1,\ldots,T_n)$, where $T_1=T_2=T$ and $T_3=\dots=T_n=T'$. We claim that $\TT$ does not contain a rainbow Hamilton cycle. Suppose for contradiction that there is such a cycle. Then, such a cycle must contain a rainbow path $P$ starting from the vertex $n$ and ending at the vertex $1$. Since there are no arcs $\arc{ji}$ with $i\le j-2$ in any color, the length of $P$ must be $n-1$ and it consists of the arcs $\arc{ji}$ with $2\le i+1=j\le n$. However, there are only $n-2$ colors where these backward arcs are there, which contradicts the fact that $P$ is rainbow. This completes the proof of \Cref{prop:tight-dir-ham-cycle}.  
\end{proof}

\paragraph{Remarks on our proof strategy.} We remind the readers that the proofs of the classical results in~\ref{every tournament has a directed Hamilton path} and~\ref{every strongly connected tournament contains a directed Hamilton cycle} are quite straightforward. Indeed, there are various ways to prove those results by using several techniques such as induction, maximality, median orders etc. Such elegant proofs do not apply to \Cref{thm:rainbow-dir-ham-path,thm:rainbow-dir-ham-cycle}. 
An ideal hope to tackle \Cref{question}, in general, is that one uses the corresponding non-transversal result as a `black box' to find a copy of $\calD$ and then color it in a rainbow way. This is indeed possible (by a greedy coloring) when the number of colors is much more than the number of edges of~$\calD$. However, in the literature, this idea of using the non-transversal result directly never succeeded to find transversal when the number of color is exactly the number of edges of~$\calD$. For example, in \cite{montgomery2022transversal}, although the authors intended to use the classical results as `black box' as much as possible to find transversal embeddings, they needed to use absorption technique and partitioned the vertex set into subsets to embed small parts at a time. 
Similarly, in our case, \Cref{thm:rainbow-dir-ham-path} is easy to prove if there is one surplus tournament in the collection $\TT$ (see, e.g., \Cref{lem:rainbow-dir-ham-path_n-cols}). However, squeezing this one extra color seems difficult, and we use the color-absorption lemma from \cite{montgomery2022transversal} to overcome it, and partition the vertex set of the collection of tournaments into smaller convenient subsets, and we will use~\ref{every tournament has a directed Hamilton path} in these subsets. For a more detailed proof sketch, see~\Cref{sec: proofsketch}. Of course, it will be interesting to see if we can find proofs of \Cref{thm:rainbow-dir-ham-path,thm:rainbow-dir-ham-cycle} directly using~\ref{every tournament has a directed Hamilton path} and~\ref{every strongly connected tournament contains a directed Hamilton cycle} and without using absorption.

\paragraph{Organization.}
The rest of the paper is organized as follows. In the next section, we start with collecting a few notations and tools that will be useful throughout the paper, and then we mention brief proof sketches of our main results. In \Cref{sec:transv-ham-paths}, we prove \Cref{thm:rainbow-dir-ham-path} and also establish a general lemma that will be useful to us later. Using this lemma along with some additional arguments, we prove \Cref{thm:rainbow-dir-ham-cycle} in \Cref{sec:rainbow-dir-ham-cycles}. 

\section{Preliminaries}\label{sec:preliminaries}

\subsection{Notations} 
For a positive integer \(n\), we write \([n]\defeq \{1,2,\dots,n\}\) and for two positive integers $a<b$, we write $[a,b]\defeq \{a,a+1,\dots, b\}$.
If we say that a result holds when ${0< \delta \ll \gamma, \beta\ll \alpha<1}$, we mean that there exist non-decreasing functions $f : (0,1] \rightarrow (0,1]$ and $g : (0,1]^2 \rightarrow (0,1]$ such that the result holds when $\gamma,\beta \leq f(\alpha)$ and $\delta \leq g(\gamma, \beta)$. We will often not explicitly calculate these functions.
We omit floors and ceilings where they are not crucial. 

\paragraph{Digraph.}
We use standard terminologies from graph theory. 
Let $\calD$ be a digraph. Denote the set of vertices of $\calD$ by $V(\calD)$ and the set of arcs in $\calD$ by $E(\calD)$. We write $e(\calD)$ to denote the number of arcs in $\calD$. 
For two digraphs $\calD_1$ and $\calD_2$, we denote the disjoint union of them by $\calD_1 \cup \calD_2$.
For a vertex $v$ of $\calD$, denote by $d_\calD^+(v)$ and $d_\calD^-(v)$ the out-degree and in-degree of $v$, respectively. Denote by $N_\calD^+(v)$ and $N_\calD^-(v)$ the out-neighborhood and in-neighborhood of $v$, respectively. 
For \(U\subseteq V(\calD)\) and \(\sigma\in\{+,-\}\), we let \(N_\calD^\sigma(v,U)\defeq N_\calD^\sigma(v)\cap U\) and \(d_\calD^\sigma(v,U)\defeq |N_\calD^\sigma(v,U)|\). We often omit the subscript \(\calD\) if it is clear from the context.
For $U\subseteq V(\calD)$, we denote by $\calD \setminus U$ the digraph induced by $V(\calD)\setminus U$.

For vertices \(u\) and \(v\) in a digraph, the arc from \(u\) to \(v\) is denoted by \(\arc{uv}\), \(\arcrev{vu}\), and we often write \(u\rarc v\) to say that there is an arc directed from \(u\) to \(v\). 
For disjoint vertex sets $X$ and $Y$, we write $E[X, Y]$ as the set of arcs directed from a vertex in $X$ to a vertex in $Y$. 
For a given digraph \(\calD\) and its disjoint vertex subsets \(X\) and \(Y\), we write \(X\outdir Y\) or \(Y\indir X\) when \(\arc{xy}\in E(\calD)\) for every \((x,y)\in X\times Y\).  

A \emph{(directed) path} of length \(\ell\) in a digraph $\calD$ is a sequence of distinct vertices $(v_1, \ldots, v_{\ell+1})$ (and often denoted by \(v_1 \cdots v_{\ell+1}\) or $v_1 \rightarrow \cdots \rightarrow v_{\ell+1}$) such that the arcs \(\arc{v_iv_{i+1}}\in E(\calD)\) for \(i\in [\ell]\). A path $v_1\rightarrow \cdots \rightarrow v_{\ell+1}$ together with an arc $v_{\ell+1}\rightarrow v_1$ is a \emph{(directed) cycle}. Paths or cycles are \emph{Hamilton} in a digraph $\mathcal{D}$ if they cover all the vertices in $\mathcal{D}$.
Consider paths $P_1,\ldots,P_m$ in a digraph $\calD$ where for every $i\in [m]$, the first and the last vertices of $P_i$ are $u_i$ and $v_i$ respectively. If for every $i\in [m-1]$, the arcs $\arc{u_iv_{i+1}}\in E(\calD)$, then we denote the concatenation of the paths $P_1,\ldots,P_m$ by $P_1 \rarc P_2 \rarc \cdots \rarc P_m$ or simply $P_1P_2\cdots P_m$. Sometimes, when the end vertex of $P_1=v_1\dots v_{\ell}$ and the start vertex of $P_2=v_{\ell}\dots v_{k}$ coincide, we write $P_1P_2$ to denote the path $v_1\dots v_k$.

\paragraph{Coloring.}

Let \(\calD\) be a digraph with an edge-coloring \(\varphi:E(\calD) \rightarrow A\) for some color set \(A\). If $\varphi$ is rainbow and uses only colors in some $C\subseteq A$, then we say \(\varphi\) is \emph{\(C\)-rainbow}. In this case, we also say \(\calD\) is \(C\)-rainbow. 
If a coloring \(\varphi\) is being considered in the context, then we write \emph{\(u\crar{c} v\)} to denote \(\varphi(\arc{uv})=c\). We also use this to denote the arc \(\arc{uv}\) colored by \(c\).
If in addition \(\varphi\) uses all the colors in \(C'\subseteq C\), then we say \(\varphi\) is a \emph{\((C,C')\)-rainbow coloring}.


\paragraph{Collection of Tournaments.}
Let \(\TT\) be a collection of tournaments. Define $|\TT|$ to be the number of tournaments in this collection. Unless stated otherwise, we assume \(\TT\) is of the form \(\{T_c:c\in C\}\) for some set \(C\). We say \(\Gamma(\TT)\defeq C\) is the \emph{color set} of \(\TT\). 
For an arc $e$ between two vertices of \(V(\TT)\), we define \(C_{\TT}(e)\defeq \{i\in \Gamma(\TT):e\in T_i\}\), which is the set of the colors which can be used to color \(e\). If $\TT$ is clear from the context, we omit the subscript.
For \(X\subseteq V(\TT)\) and $A\subseteq \Gamma(\TT)$, we define the \emph{vertex-induced collection} $\TT[X] \defeq\{T[X]:T\in\TT\}$ and the \emph{color-induced collection} $\TT_A \defeq \{T_i: i\in A\}$. This naturally yields $\TT_A[X] = \TT[X]_A$.

Let $\TT = \{T_1, \dots , T_m\}$ be a collection of tournaments. For \(\gamma\in(0,1]\), we define $\TT^{\gamma}$ as the digraph on the vertex set $V(\TT)$ with the arc set
\[
    \left\{\arc{uv} :  \left|\{i\in[m]:\arc{uv}\in E(T_i)\}\right|\ge\gamma m\right\}.
\]
A simple pigeonhole principle implies that \(\TT^{\gamma}\) contains at least one tournament as a subdigraph when $\gamma \leq 1/2$.
We write $\TT_A^{\gamma} = (\TT_A)^{\gamma}$ to make the order of subscripts and superscripts clear. 

We frequently use the following straightforward observations. 
For a given collection $\TT$ of tournaments,
\(0<\alpha\le\beta\le\frac{1}{2}\) and \(C_1\subseteq C_2\subseteq\Gamma(\TT)\), we have the following.
\begin{enumerate}
    \item\( \TT^{\beta}\subseteq \TT^{\alpha}.\) \label{stat:aux_rel_1}
    \item \( \TT_{C_2}^{\beta} \subseteq \TT_{C_1}^{\alpha}\) when \((1-\alpha)|C_1|\geq  (1- \beta)|C_2|\).\label{stat:aux_rel_2} 
    \item \( \TT_{C_1}^{\beta} \subseteq \TT_{C_2}^{\alpha}\) when \(\beta|C_1|\geq  \alpha|C_2|\).\label{stat:aux_rel_3}
\end{enumerate}

\subsection{\texorpdfstring{\(\Hpart\)}{\textbf{H}}-partition}
We will partition the vertex set of a given tournament in a desirable way to execute Step~1 of the proof sketch described in \Cref{sec: proofsketch}. For convenience, we give a name to such partitions as follows.
\begin{definition}
Let $0\le \gamma \le 1$ and $r$, $\ell$ be positive integers. Let \(T\) be a tournament.
A tuple \((W_1,\dots, W_r, w_1,\dots, w_{r-1})\) of disjoint vertex subsets \(W_1,\dots, W_r\subseteq V(T)\) and distinct vertices \(w_1,\dots, w_{r-1}\) in \(V(T) \setminus \bigcup_{i\in [r]} W_i\) is an \emph{$\Hpart(\ell,\gamma)$-partition} if the followings hold.
\begin{enumerate}[(1)]
    \item \((\bigcup_{i\in [r]} W_i)\cup\{w_1,\dots,w_{r-1}\}=V(T)\). 
    \item \( \gamma \ell \leq |W_i| \leq \ell \) for each \(i\in [r]\). 
    \item  \(W_{i}\outdir \{w_{i}\} \outdir W_{i+1}\) for each \(i\in [r-1]\). 
\end{enumerate}
\end{definition}

In the above definition, the edges between the vertex $w_i$ and $W_i\cup W_{i+1}$ are often referred to as \emph{intermediate edges}.
Note that an \(\Hpart(\ell,\gamma_1)\)-partition is an \(\Hpart(\ell,\gamma_2)\)-partition whenever \(0<\gamma_2\le\gamma_1\le 1\).
The following lemma proves the existence of an $\Hpart(\ell,\gamma)$-partition.

\begin{lemma}
\label{lem:H-partition}
    Let \(0 < \gamma\le \frac{1}{6}\) and \( \ell,n\) be positive integers with \(3\leq \ell\leq n\). Let \(T\) be a tournament of order \(n\). Then, \(T\) has an \(\Hpart(\ell ,\gamma)\)-partition.
\end{lemma}
\begin{proof}
Note that the statement is clear when $\ell=n$ as \(V(T)\) is a trivial \(\Hpart(\ell ,\gamma)\)-partition. Assume that \(\ell\ge 3\) is the largest possible number such that \(T\) does not have a \(\Hpart(\ell ,\gamma)\)-partition. As $\ell<n$, by the maximality of $\ell$, the tournament $T$ has an \(\Hpart(\ell+1 ,\gamma)\)-partition. Among all \(\Hpart(\ell+1 ,\gamma)\)-partitions, choose one \(P=(W_1,\dots, W_t, w_1,\dots, w_{t-1})\) such that the number of $i\in [t]$ satisfying $|W_i|=\ell+1$ is the smallest.
If every $i\in [t]$ satisfies \(|W_i| <\ell+1\), then it is also an \(\Hpart(\ell ,\gamma)\)-partition, a contradiction. 
Hence, we may assume that there exists $i_0\in [t]$ with $|W_{i_0}|=\ell+1$. 
We use the following claim.

\begin{claim} \label{clm:num-of-small-indeg-vtxs_small}
    For a tournament \(T\) and \(d\ge0\), there are at most \(2d+1\) vertices whose in-degree (resp.\ out-degree) is at most \(d\).
\end{claim}
\begin{claimproof}
    Let \(V(T)=\{v_1,\dots,v_n\}\) be an enumeration of the vertices of \(T\) which maximizes \(|\{\arc{v_iv_j} : i<j\}|\) (this is called a median order). As noted in \cite[Section~1]{havet2000median}, the in-degree of \(v_j\) in \(\{v_1,\dots,v_{j-1}\}\) is at least \((j-1)/2\) for each \(j\ge1\), so the vertices \(v_j\) with \(j\ge2d+2\) have in-degree at least \(d+1\). Thus there are at most \(2d+1\) vertices whose in-degree is at most \(d\). The out-degree case follows by the dual argument.
\end{claimproof}

By \Cref{clm:num-of-small-indeg-vtxs_small}, we can partition \(W_{i_0}\) into three sets \( W^{-}_{i_0}, \{v\}, W^{+}_{i_0}\) such that \( W^{-}_{i_0} \outdir \{v\} \outdir W^{+}_{i_0}\) and
\(|W^-_{i_0}|, |W^+_{i_0}| \geq \frac{1}{6} |W_{i_0}| \geq \gamma (\ell +1)\).
Now consider the partition 
\begin{align*}
    P' &= (W'_1,\dots, W'_{t+1}, w'_1,\dots, w'_{t}) \\
    &= (W_1,\dots, W_{i_0-1}, W^{-}_{i_0},W^{+}_{i_0}, W_{i_0+1},\dots, W_t, w_1,\dots, w_{i_0-1}, v, w_{i_0+1},\dots, w_{t-1}).
\end{align*}
We claim that this is an \(\Hpart(\ell+1 ,\gamma)\)-partition having one smaller number of $i\in [t+1]$ with $|W'_i|=\ell+1$ compared to $P$.
Indeed, since $W_{i_0}$ is partitioned, it is clear that $P'$ contains less number of vertex sets of size $\ell+1$. Also, it is clear that $P'$ is an \(\Hpart(\ell+1,\gamma)\)-partition, which is a contradiction.
This completes the proof of \Cref{lem:H-partition}.
\end{proof}

\subsection{Color absorption}

The following lemma in \cite{montgomery2022transversal} provides a color absorber. This concept of color absorber was very useful for finding a spanning transversal copy in a given collection of graphs, such as $F$-factors and bounded degree trees~\cite{montgomery2022transversal}, powers of Hamilton cycles~\cite{gupta2023general}, bounded degree graphs with sublinear bandwidth~\cite{chakraborti2023bandwidth}, and Hamilton cycles and other spanning structures in hypergraphs~\cite{cheng2021transversal}.
Again, this allows us to find a color absorber in an appropriate collection of tournaments.

\begin{lemma}[{\cite[Lemma~3.3]{montgomery2022transversal}}]\label{lem:absorber_prelim}
    Let $\alpha \in (0, 1)$, let $n, m$ and $\ell\geq 1$ be integers satisfying $\ell\leq \alpha^7 m/ 10^5$ and $\alpha^2 n \geq 8m$. Let $H$ be a bipartite graph on vertex classes $A$ and $B$ such that $|A| = m$, $|B| = n$ and, for each $v\in A$, $d_H(v)\geq \alpha n$.
    
    Then, there are disjoint subsets $B_0, B_1\subseteq B$ with $|B_0| = m-\ell$ and $|B_1|\geq \alpha^7 n/ 10^5$, and the following property. Given any set $U\subseteq B_1$ of size $\ell$, there is a perfect matching between $A$ and $B_0\cup U$ in $H$.
\end{lemma}

Using this, we can easily deduce the following color absorption lemma, which fits our setting.

\begin{lemma}[Color absorption lemma] \label{lem:absorber}
    Let \(0<1/n\ll\gamma\ll\beta\ll\alpha \leq 1/2\). Let \(H\) be a digraph with \(\beta n\leq e(H)\leq (\beta+ \frac{1}{2}\gamma ) n\), and \(\TT\) be a collection of tournaments with \(|V(\TT)|=n\) and \(|\TT|=m\ge\alpha n\). Let \(S\) be a copy of \(H\) in \(\TT^{\alpha}\).
    
    Then, there exist disjoint sets \(A,C\subseteq[m]\), with \(|A|=e(H)-\gamma m\) and \(|C|\ge 10\beta m\) such that the following property holds. Given any subset \(C'\subseteq C\) of size \(\gamma m\), there is a rainbow coloring of \(S\) in \(\TT\) using colors in \(A\cup C'\).
\end{lemma}
\begin{proof}
    Let \(K\) be the bipartite graph with vertex classes \(E(S)\) and \([m]\), where \((e,i)\) is an edge if \(e\in T_i\). Then \Cref{lem:absorber_prelim} applies with \((\ell,m,n)=(\gamma m, e(H), m)\). This yields disjoint \(A,C\subseteq[m]\) with \(|A|=e(H)-\gamma m\) and \(|C|\ge 10\beta m\), such that, for any \(C'\subseteq C\) of size \(\gamma m\) there is a perfect matching between \(E(S)\) and \(A\cup C'\). Such a matching corresponds to an (\(A\cup C'\))-rainbow coloring of \(S\) in \(\TT\), as required.
\end{proof}

\subsection{Proof sketches} \label{sec: proofsketch}
Here, we outline the major steps involved in the proofs of our results. We start with \Cref{thm:rainbow-dir-ham-path}.
Firstly, when we have many spare colors, it is easy to find a rainbow Hamilton path. Indeed, if $\TT$ has at least \(2n\) tournaments, then one can find a Hamilton path $P$ in \(\TT^{1/2}\) and then $P$ can be greedily colored in a rainbow way as $\frac{1}{2}\cdot 2n > n-1$ colors are available for each edge. 
In fact, as we will see in \Cref{lem:rainbow-dir-ham-path_n-cols}, one can find a rainbow Hamilton path even when only one spare color is provided. However, it is not trivial to find a transversal Hamilton path when the number of colors is exactly $n-1$.

To overcome the difficulties, we decompose the path of order $n$ in smaller subpaths, and iteratively apply the above to embed paths into these smaller parts. 
This yields a rainbow linear forest of at least $(1-o(1))n$ edges. To convert this into a Hamilton path, we connect them and use the color absorption lemma. This idea is elaborated below. We will choose constants $\mu,\gamma,\beta$ so that $0 < 1/n\ll \mu \ll \gamma \ll \beta \ll 1$ holds.
    \bigskip

    \noindent {\bf Step 1: Partition the vertex set.}
    We first consider a tournament $T\subseteq \TT^{1/2}$. Then, consider an $\Hpart(\mu n, \gamma)$-partition \((W_1,\dots, W_r, w_1,\dots, w_{r-1})\) of \(T\). 
    Let $E_{i,1}:= \{\arc{vw_i} : v\in W_i\}$ and 
    $E_{i,2}:= \{\arc{w_iv} : v\in W_{i+1}\}$. 
    We plan to later find Hamilton paths $P_i$ in each of the $W_i$'s which yield a Hamilton path $P_1 e_{1,1} e_{1,2} P_2 e_{2,1} e_{2,2} P_3 \cdots P_r$ of $T$ for some $e_{i,j}\in E_{i,j}$. In the next steps, we illustrate how to obtain such paths so that one can find a rainbow coloring of the final Hamilton path.
    \bigskip
    
    \noindent {\bf Step 2: Partition the color set and find a color absorber.}
    We set aside a set \(D\) of at most $\mu^2 n$ colors such that for any collection \( \{e_{i,j}: i\in [r-1],\, j\in [2]\}\) of one arc from each $E_{i,j}$, we can find a \(D\)-rainbow coloring of these arcs. 
    Let $t\in [r]$ be such that $\sum_{i\in [t]} (|W_i|-1) = \beta n + o(n)$.
    For each $i\in [t]$, let $P_i$ be a Hamilton path of $T[W_i]$. Next, utilize \Cref{lem:absorber} to find disjoint \(A,C\subseteq \Gamma(\TT)\setminus D\) such that $|A| = \beta n - \gamma n$ and $|C| = 10\beta n$ and for any \(C'\subseteq C\) of size \(\gamma n\), we can find an (\(A\cup C'\))-rainbow coloring of the arcs in \(\bigcup_{i\in [t]} P_i\). Denote \(B\defeq \Gamma(\TT)\setminus(A\cup C\cup D)\). 
    \bigskip
    
    \noindent {\bf Step 3: Use most of the colors in {\boldmath \(B\)} and color most intermediate arcs.}
    We first choose a number $\tau\in [t+1,r-1]$. For each $i\in [t+1,r]\setminus\{\tau\}$ one by one, we consider the collection of tournaments $\TT_{C'}[W_i]$ where $C'\subseteq B\cup C$ is the set of current available colors.
    Then we find a Hamilton path in  $\TT_{C'}[W_i]$ and color the arcs of this path. As $|W_i|$ is much smaller than $|C'|$, we can greedily color the arcs. We can repeat this for all $i\in [t+1,r]\setminus\{\tau\}$.
    Furthermore, we can ensure that all colors of $B$ are used, possibly except at most $r$ colors.
    This procedure can be done using \Cref{lem:rainbow-dir-ham-path_n-cols}.
    After this step, suppose $C^*,B^*$ denote the remaining unused colors in $C,B$ respectively.

    Let $u_i,v_i$ denote the starting and ending vertices of the Hamilton paths considered inside $W_i$. We now $D$-rainbow color the intermediate arcs $\{\arc{w_i u_{i+1}} : i\in [r-1]\}\cup \{\arc{v_i w_i}: i\in [r-1]\}$. Let $D^*$ be the remaining unused colors in $D$. 
    \bigskip
    
    \noindent {\bf Step 4: Absorb the colors in {\boldmath \(B^*\cup D^*\)} using the colors in {\boldmath \(C^*\)}.}
    Since the number of colors in ${B^*\cup D^*}$ is small enough, we can embed a rainbow Hamilton path in $W_\tau\cup \{w_{\tau-1}, w_{\tau}\}$ that starts from the vertex $w_{\tau-1}$ to $w_{\tau}$ using up all the colors in $B^*\cup D^*$ and some colors in $C^*$. This is done by applying \Cref{lem:embed-linear-fixed-color}.
    Finally, we use the remaining unused colors in \(C\) along with those in \(A\) to color the arcs in $P_i$ with $i\in [t]$, which completes the desired transversal Hamilton path in $\TT$.
    \bigskip

To prove \Cref{thm:rainbow-dir-ham-cycle}, we again need to use similar arguments to find a Hamilton path in $\TT$ along with some extra arguments to close the path to a Hamilton cycle. 
Thus, we extract a lemma (see~\Cref{lem:rainbowDHP}) capturing these similar arguments that can be directly applicable to both \Cref{thm:rainbow-dir-ham-path,thm:rainbow-dir-ham-cycle}.

Similar to before, consider a tournament $T\subseteq \TT^{1/2}$.
By \Cref{lem:H-partition}, we first find an $\Hpart(\mu n, 1/6)$-partition $(W_0, W_1, \dots, W_{r+1}, w_0, \dots, w_{r})$ of $T$. 
We next find a Hamilton path $P_0$ of $T[W_0\cup \{w_0\}]$ from some vertex $x$ to $w_0$ and a Hamilton path $P_{r+1}$ of $T[W_{r+1}\cup \{w_r\}]$ from $w_r$ to some vertex $y$. 
If $\arc{yx}\in T_i$ for some $i\in [n]$, then we can utilize the $\Hpart$-partition to find a $[n]\setminus\{i\}$-rainbow Hamilton path from $x$ to $y$, yielding a desired rainbow Hamilton cycle. So we may assume that $\arc{xy}\in T_i$ for all $i\in [n]$.

If there are many internally disjoint short rainbow paths from $y$ to $x$, then we can choose one such short path $P$ from $y$ to $x$ that does not use any vertex in $W_0\cup W_{r+1} \cup \{w_0, \dots, w_{r}\}$.
After removing the vertices used in $P$, the modified partition \((W_0,W_1\setminus V(P),W_2\setminus V(P),\dots,W_r\setminus V(P),W_{r+1}, w_0,\dots,w_r)\) still forms an $\Hpart(\mu n, 1/10)$-partition of $T\setminus V(P)$.
This fact then helps us to carry through the same 4 steps to find an appropriate rainbow Hamilton path $P'$ in $V(\TT)\setminus (W_0\cup W_{r+1}\cup V(P))$ from $w_0$ to $w_r$, yielding a desired rainbow Hamilton cycle.

Hence, we are left with the case that
$\arc{xy}\in T_i$ for all $i\in [n]$ and there are not many internally disjoint short rainbow paths from $y$ to $x$. These two properties will be useful.

We consider a longest rainbow path $P^*=x_1\to \dots \to x_k$ from $y=x_1$ to $x=x_k$ and let $D$ be the set of colors not used by the path. If this is a Hamilton path, then we can close it to a rainbow cycle as $\arc{xy}\in T_c$ for the remaining color $c$.
If not, the maximality of this path provides information between the leftover vertices $z$ and $V(P^*)$. For example, if $\arc{zx_i}\in T_j$ for some $j\in D$, then $\arc{x_{i-1}z}\notin T_{j'}$ for $j'\in D\setminus \{j\}$ as otherwise we can find a longer path $x_1\to \dots x_{i-1} \crar{j'} z \crar{j} x_i \to \dots \to x_k $. This fact together with the lack of many internally disjoint short rainbow paths from $y$ to $x$ allows us to obtain patterns on the directions of the arcs between the vertices outside $P^*$ and the vertices in $V(P^*)$.

Once we have the patterns, some clever choices of maximality together with the strong connectedness of the tournaments yield a rainbow Hamilton path that can be closed into a rainbow Hamilton cycle.
Details are provided in \Cref{sec:rainbow-dir-ham-cycles}.

\section{Transversal Hamilton paths}
\label{sec:transv-ham-paths}
 
In this section, we prove \Cref{thm:rainbow-dir-ham-path}. 
The following lemma shows that one can find a transversal Hamilton path with one additional color. 

\begin{lemma}
\label{lem:rainbow-dir-ham-path_n-cols}
    For any collection of tournaments $\TT$ with $|V(\TT)| = |\TT| = n$, there is a rainbow Hamilton path.
\end{lemma}
\begin{proof}
    Take a longest rainbow path $P=u_1 \rarc u_2 \rarc \cdots \rarc u_r$ in $\TT$. Suppose  $|V(P)| < n$. Then, there is at least one vertex $v \in V(\TT)\setminus V(P)$ and at least two unused colors, say $1$ and $2$. We have \(\arc{u_1v}\in E(T_1)\) since otherwise there is a rainbow path \(v\crar{1} P\) longer than \(P\), a contradiction to the maximality of \(P\). Similarly, we have $\arc{u_1v} \in E(T_2)$.
    
    We claim that \(\arc{u_tv}\in E(T_i)\) holds for all \(t\in [r]\) and \(i\in \{1,2\}\). If not, let \(2\le t'\le r\) be the smallest index where \(\arcrev{u_{t'}v}\in E(T_j)\) for some \(j\in\{1,2\}\). By the minimality of \(t'\), we have \(\arc{u_{t'-1}v}\in E(T_{3-j})\), which gives a rainbow path \(u_1 \rarc \cdots \rarc u_{t'-1} \crar{3-j} v \crar{j} u_{t'} \rarc \cdots \rarc u_r\) longer than \(P\), a contradiction. Thus the claim holds, whence in particular \(\arc{u_rv}\in E(T_i)\) for \(i\in \{1,2\}\). However, this gives $P\crar{1}v$ which is again a rainbow path longer than \(P\), a contradiction. Therefore, $|V(P)| = n$ and $P$ is a rainbow Hamilton path.
\end{proof}

Having more additional colors, instead of just one, allows us to find a rainbow Hamilton path in a robust way: we can designate some colors to appear on the path. The following lemma allows us to ensure that one fixed color must appear on the path. Using this, we can even make sure that linearly many preselected colors appear on the path.

\begin{lemma}\label{lem:embed-one-fixed-color}
    Let $\TT$ be a collection of tournaments with $|V(\TT)| = n\ge 2$ and $|\TT| \geq 2n$. 
    For each $i\in \Gamma(\TT)$, there is a rainbow Hamilton path in \(\TT\) using an arc of $T_i$, except for the case where  the tournament $T_i$ is a directed triangle with $n=3$, and the other tournaments in $\TT$ are directed triangles with the opposite orientation.
\end{lemma}
\begin{proof}
    Assume without loss of generality that \(\Gamma(\TT)=[2n]\) and \(i=1\). We use induction on \(n\). The base cases when \(n=2\) or \(3\) are clear. Let \(n\ge 4\) and assume the statement holds for smaller \(n\) and $\TT$ is the smallest counterexample to the statement. We call a collection of tournaments exceptional if one tournament is a directed triangle and all the others are directed triangles with the opposite orientation.
    Choose a tournament $T\subseteq \TT^{1/2}$.

    \begin{claim}\label{cl: exceptional}
    For any vertex $v\in V(\TT)$, if $d^+_T(v)\geq 2$, then the collection $\TT[N^+_T(v)]$ is exceptional.
    Similarly, if $d^-_T(v) \geq 2$, then the collection $\TT[N^-_T(v)]$ is exceptional.
    \end{claim}
    \begin{claimproof}
        Suppose that \(v\) has out-degree at least two and \(\TT[N_T^+(v)]\) is not exceptional. By the induction hypothesis, there exists a rainbow directed path \(P = v_1\rarc \cdots \rarc v_t \) with \(V(P) =  N_T^+(v)\) using one arc in \(T_1\).
        Since $\arc{vv_1}\in T\subseteq \TT^{1/2}$, there are at least $\frac{1}{2}\cdot 2n = n$ colors containing $\arc{vv_1}$. As at most \(d_T^+(v)-1 \le n-2\) colors are used in \(P\), we can choose a color to extend \(P\) to a rainbow path \( P'= v\rarc v_1 \rarc \cdots \rarc v_t\).
        
        If $d^{-}_T(v)=0$, then \(P'\) is the desired rainbow Hamilton path of $\TT$, a contradiction. Otherwise, letting $I \subseteq [2n]$ be the collection of colors not used by \(P'\), we have \(|I| \geq 2n - (n-1) \geq n+1\) colors which are not yet used. By applying \Cref{lem:rainbow-dir-ham-path_n-cols} to $\TT_I[N^-_T(v)]$, we obtain a (possibly empty) rainbow Hamilton path \(Q = u_1\rarc \cdots \rarc u_s\) on the vertex set $N^-_T(v)$. As at most \((d_T^+(v)-1)+(d_T^-(v)-1)+1=n-2\) colors are used in $P'\cup Q$, we can still choose a color \(j\in [2n]\) such that $\arc{u_s v} \in T_j$ and \(j\) is not used in \(P'\cup Q\). 
        This yields a desired rainbow Hamilton path $Q\rarc P'$ of $\TT$ using one arc from $T_1$, a contradiction. The similar argument works when \(d_T^-(v)\ge 2\) and \(\TT[N_T^-(v)]\) is not exceptional.
    \end{claimproof}
    
Now we use this claim to derive a contradiction. 
If $n \geq 8$, then there must be a vertex $v$ with $d^+_T(v)\geq 4$, whence $\TT[N^+_T(v)]$ is not exceptional, a contradiction to \Cref{cl: exceptional}.
If $4\leq n \leq 6$, then it is easy to check that there must be a vertex $v$ such that either $d^{+}(v)\notin \{1,3\}$ or $d^-(v) \notin \{1,3\}$, which yields a contradiction to \Cref{cl: exceptional}.

The only remaining case is when $n=7$, and $T$ is a $3$-regular tournament. Choose $v\in V(\TT)$. 
\Cref{cl: exceptional} implies that $\TT[N^+_T(v)]$ and $\TT[N^-_T(v)]$ are both exceptional.
Then, $T[N^+_T(v)]$ is a directed triangle $a\rarc b \rarc c \rarc a$ and moreover $T_1[N^+_T(v)]$ is the directed triangle $c\rarc b \rarc a \rarc c$.
Also, $T[N^-_T(v)]$ is a directed triangle $x\rarc y \rarc z \rarc x$ and moreover $T_1[N^-_T(v)]$ is the directed triangle $z\rarc y \rarc x \rarc z$. 
As $T$ is $3$-regular, we can assume that $\arc{ax} \in T$.
Then, as every arc in $T$ belongs to at least $n$ colors, we can find a rainbow direct Hamilton path $z\crar{1} y \rarc v \rarc b \rarc c \rarc a \rarc x$ by greedily choosing colors of all arcs other than the first one.
Again, this contradiction completes the proof.
\end{proof}

We next show that for any linearly many specified colors, we can find a rainbow Hamilton path using those colors. 
\begin{lemma}
\label{lem:embed-linear-fixed-color}
    Let $\TT$ be a collection of tournaments with $|V(\TT)| = n\geq 25$ and $|\TT| =m \geq 4n$. Let $B\subseteq \Gamma(\TT)$ with $|B| \le \frac{n}{25}$. 
    Let $u,v$ be vertices in $V(\TT)$ such that for each $w\in V(\TT)\setminus\{u,v\}$, the arc $\arc{uw}$ is in $T_i$ for at least two choices of $i\in B$ and the arc $\arc{wv}$ is in $T_j$ for at least two choices  of $j\in B$.
    Then, $\TT$ has an $(\Gamma(\TT),B)$-rainbow Hamilton path starting from $u$ and ending at $v$.
\end{lemma}
\begin{proof}
    Assume without loss of generality that \(\Gamma(\TT)=[4n]\). Let $V \defeq V(\TT)\setminus\{u,v\}$.
    Choose \(T\subseteq \TT^{1/2}[V]\), and apply \Cref{lem:H-partition} to find an \(\Hpart(24,\frac{1}{6})\)-partition $(W_1, \dots, W_r, w_1, \dots, w_{r-1})$ of \(T\). Note that we have \(n/25\le r\le n\) and \(|W_i|\geq 4\) for each \(i\). Assume without loss of generality that $B\subseteq [r]$.
    Partition the color set \([r+1,4n]\) into \(C_1,\dots, C_r\) where \(|C_i|\ge 2|W_i|\) for each \(i\).
    
    For $j\in \{1,r\}$, use \Cref{lem:embed-one-fixed-color} to obtain a $C_i$-rainbow Hamilton path $P_i$ in $W_i$ starting from a vertex $u_j$ and ending at a vertex $v_j$. We choose two colors, say $1$ and $r$, so that $\arc{uu_1}\in T_1$ and $\arc{v_rv} \in T_r$. 
    
    For each $i=2,\dots, r-1$, we apply \Cref{lem:embed-one-fixed-color} to find a \((C_i\cup \{i\},i)\)-rainbow Hamilton path \(P_i\) of \(\TT[W_i]\) which starts from a vertex $u_i$ and ends at a vertex $v_i$. This uses all the colors in \(B=[r]\). We have used at most \(n\) colors, so for each arc in \(T\) there are at least \(\frac{1}{2} \cdot 4n-n=n\) unused colors. Thus we can greedily choose distinct colors for \(\{\arc{v_iw_i},\, \arc{w_i u_{i+1}}:1\le i\le r\}\) to complete a rainbow Hamilton path \(u \rightarrow P_1 \rightarrow w_1 \rightarrow P_2 \rightarrow \dots \rightarrow w_r \rightarrow P_{r+1} \rightarrow v\) of \(\TT\), which uses all the colors in \(B\) as desired.
\end{proof}

\subsection{Main lemma}
In this subsection, we prove the following lemma. It is straightforward to derive \Cref{thm:rainbow-dir-ham-path} from this lemma.
\begin{lemma}
\label{lem:rainbowDHP}
    Let \(0<\frac{1}{n}\ll\mu\ll \gamma, \alpha\le 1\). Let \(\TT\) be a collection of tournaments with \(|V(\TT)| = n\) and \(|\TT| = n-1\). 
    Let \(T\subseteq \TT^{\alpha}\) be a tournament. Let \(w_0,w_r\in W\), and \((W_1,\dots,W_r,w_1,\dots,w_{r-1})\) be an \(\Hpart(\mu n,\gamma)\)-partition of \(T\setminus\{w_0,w_r\}\).
    Suppose \(w_0\outdir W_1\) and \(W_r\outdir w_r\) in~$T$.
    Then, there is a rainbow Hamilton path in \(\TT\) from \(w_0\) to \(w_r\).
\end{lemma}

We first show that \Cref{lem:rainbowDHP} implies \Cref{thm:rainbow-dir-ham-path}.

\begin{proof}[Proof of \Cref{thm:rainbow-dir-ham-path}]
Let \(0<\frac{1}{n}\ll\mu\ll \gamma\le 1/6\). Fix a tournament $T\subseteq  \TT^{1/2}$. By \Cref{lem:H-partition}, $T$ has an $\Hpart(\mu n, \gamma)$-partition $(W_0,\dots, W_{r+1}, w_0,\dots, w_r)$ for some $r\in \mathbb{N}$. Choose a Hamilton path $P_0$ of $T[W_0]$ and a Hamilton path $P_{r+1}$ of $T[W_{r+1}]$. Denote the last vertex of $P_0$ by $u$ and the first vertex of $P_{r+1}$ by $v$. Since the arcs of $P_0$ and $P_{r+1}$ and the arcs $\arc{uw_0}$ and $\arc{w_rv}$ lie in at least $(n-1)/2> |E(P_0)|+|E(P_{r+1})|+2$ different colors, we can greedily choose colors of the arcs in $E(P_0)\cup E(P_{r+1})\cup \{\arc{uw_0},\, \arc{w_rv}\}$ in a rainbow way.
Let $C\subseteq \Gamma(\TT)$ be the set of all the remaining colors, then $T$ is still a subtournament of $\TT_C^{1/4}$. 
By considering the collection $\TT_C[V(\TT)\setminus (W_0\cup W_{r+1}\cup \{w_0,w_r\})]$ 
together with an $\Hpart(\mu n, \gamma)$-partition \((W_1,\dots, W_r,w_1,\dots, w_{r-1}\)) of $T\setminus (W_0\cup W_{r+1}\cup \{w_0,w_r\})$, Lemma~\ref{lem:rainbowDHP} applied with $\alpha=1/4$ yields a $C$-rainbow Hamilton path from $w_0$ to $w_r$. This together with the precolored paths $P_0$ and $P_{r+1}$ and the arcs $\arc{uw_0}$ and $\arc{w_rv}$ yields a rainbow Hamilton path in~$\TT$.
\end{proof}

\begin{proof}[Proof of \Cref{lem:rainbowDHP}]
Note that an $\Hpart(\mu n, \gamma)$-partition is also an $\Hpart(\mu n, \gamma')$-partition for all $\gamma' < \gamma$. Hence, it is enough to prove the lemma further assuming $\mu\ll \gamma \ll \alpha$. We assume this in the rest of the proof.
We further choose a number $\beta$ satisfying $0< 1/n\ll \mu \ll \gamma \ll \beta \ll \alpha < 1$. We take a few steps to finish the proof, and the steps loosely follow the proof sketch given in \Cref{sec:preliminaries}.

\paragraph{Step 1. Reserve a color set $D$ for the intermediate arcs in the $\Hpart$-partition.}
We start with setting aside a set $D$ of colors which we will later use for the arcs incident with the vertices $w_0,w_1,\ldots,w_{r-1},w_r$.
For each $i\in [r]$, let $\calA^-_i =  \{C_{\TT}(\arc{vw_i}) : v\in W_i\}$ be the multi-collection of the sets of available colors for the arcs between the vertices in $W_i$ and $w_i$, and let $\calA^{+}_i = \{C_{\TT}(\arc{w_{i-1}v}): v\in W_i\}$ be the multi-collection of the sets of available colors for the arcs between $w_{i-1}$ and the vertices in $W_i$. Note that each set $A$ in $\calA^-_i\cup \calA^+_i$ has size at least $\alpha n$.
For each element in $\Gamma(\TT)$, we include it in the set $D$ independently at random with probability $\frac{ 20\mu^{-2} r \log n}{n}$. Then, the standard Chernoff bound yields that both of the followings simultaneously hold with a positive probability.
\begin{proplist}
    \item\label{eq: SDR1} $|D|\leq 100 \mu^{-2} r\log n < \mu^2 n$. 
    \item\label{eq: SDR2} For each $i\in [r-1]$ and each set $A \in \calA^-_i\cup \calA^+_i$, we have $|D\cap A|  > 2r$. 
\end{proplist}
Let us fix one such set $D$.

\paragraph{Step 2. Setting up a color absorber.} 
Let $t\leq r-2$ be an index such that 
$$\beta n\leq \sum_{i\in [t]} (|W_i|-1) \leq (\beta+ \mu) n.$$
Such a number $t$ exists as $|W_i|\leq \mu n$ and $\mu \ll \beta$.
For each $i\in [t]$, take a spanning path $P_i$ in $T[W_i]$. 
Let $\tau$ be an arbitrary element in $[r]$ with $t<\tau < r$. 
We partition $[r]$ into three sets $L_1$, $L_2$, $L_3$ as follows:
\[
    L_1 \defeq [t], \quad
    L_2 \defeq [r]\setminus (L_1\cup \{\tau\}), \quad
    L_3 \defeq \{\tau\}.
\]
Let $Q_1$ be the union of the paths $\bigcup_{i\in L_1} P_i$. 
Every arc $e$ of $T$ satisfies $|C(e)\setminus D| \geq \alpha n - \mu^2 n \geq \alpha n/2$. 
So we have $Q_1\subseteq T\subseteq \TT_{\Gamma(V(\TT))\setminus D}^{\alpha/2}$.
Hence, \Cref{lem:absorber} with $\alpha /2$ playing the role of $\alpha$ ensures a partition $A\cup B\cup C$ of $\Gamma(V(\TT))\setminus D$ satisfying the following.
\begin{proplist}
    \item $\beta n - \gamma n \leq |A| = e(Q_1) - \gamma n\leq  \beta n - \frac{1}{2} \gamma n$.\label{eq: size of A}
    \item $|C| \geq 10 \beta n$.\label{eq: size of B}
    \item For any subset $C'\subseteq C$ of size $e(Q_1) - |A|$, there is a rainbow coloring of $Q_1$ in $\TT$ using colors in~$A\cup C'$.\label{eq: color absorbing}
\end{proplist}
By \ref{eq: SDR1}, \ref{eq: size of A}, \ref{eq: size of B}, and the fact that $|A|+|B|+|C|+|D|=n-1$, we have 
\begin{align}
  |B|+|C| &=  n-1 - e(Q_1) + \gamma n -|D| \geq n - e(Q_1) + \frac{1}{2}\gamma n, \label{eq:lower bound on B and C}    
  \\|B| &\leq n- e(Q_1) + \gamma n -|C| \leq n-e(Q_1) - 9\beta n. \label{eq:upper bound on B}
\end{align}

\paragraph{Step 3. Use most of the colors in $B$ and color most intermediate arcs.} 
Next, we will choose a Hamilton path in $W_i$ for each $i\in L_2$ and choose colors for the arcs in those paths. While doing that, we aim to exhaust most of the colors in $B$ while using some additional colors in $C$. 
By \eqref{eq:lower bound on B and C}, \eqref{eq:upper bound on B}, the fact $|W_{\tau}|\leq \mu n$ and the fact that $e(Q_1)\geq \beta n$, we know that $M: = \sum_{i\in L_2} |W_i|$ satisfies 
  \begin{align*}
    (|B| + |C|) - \frac{1}{2}\gamma n
    \ge n - e(Q_1)
    \ge M
    \ge n - e(Q_1) - \mu n - (r+1+t)
    \ge n -e(Q_1) - \beta n
    \ge |B|.
  \end{align*} 
Using this, we choose an arbitrary subset $C_1\subseteq C$ with $|B\cup C_1|=M$. Furthermore, the set $C\setminus C_1$ satisfies
$$|C\setminus C_1| = |B| + |C| - M \geq \frac{1}{2}\gamma n.$$

Now, we partition the set $B\cup C_1$ into 
$\{B_i: i\in L_2\}$ with $|B_i|=|W_i|$ for each $i\in L_2$.
For each $i\in L_2$, we apply \Cref{lem:rainbow-dir-ham-path_n-cols} to the collection $\TT_{B_i}[W_i]$. This yields a $B_i$-rainbow Hamilton path in $\TT[W_i]$ using all colors in $B_i$ except exactly one color $b_i$, say. Let $Q_2$ be the union of such paths.
Let $B' = \{b_i :i\in L_2\}$ be the set of at most $r$ unused colors.
Let $B^* \defeq B\cap B'$ and $C^* \defeq (C\setminus C_1)\cup (C\cap B')$. Then, $B^*$ and $C^*$ are the set of the unused colors in $B$ and $C$, respectively. Moreover, 
\begin{equation} \label{eq:size of B and D star}
|B^*| \leq r \;\;\;\;\; \text{and} \;\;\;\;\; |C^*| \geq \frac{1}{2} \gamma n.
\end{equation}

We currently have a union of paths $Q_1\cup Q_2$ where the arcs of $Q_1$ are not colored, but the arcs in $Q_2$ have received pairwise distinct colors.
Now, for each $i\in [r]\setminus\{\tau\}$, we have a path-component of $Q_1\cup Q_2$ within $W_i$, which starts from $u_i$ and ends at $v_i$, say.
Now we will connect the paths using the vertices $w_i$. 
Consider the sets $E$ of arcs $\{\arc{w_{i-1} u_i} : i\in [r]\setminus \{\tau\}\}\cup \{\arc{v_i w_i}: i\in [r]\setminus \{\tau\}\}$.
For each arc $e\in E$, we greedily choose a color in $C(e)\cap D$ and color it with the color, ensuring the arcs in $E$ receive pairwise distinct colors. By \ref{eq: SDR2}, we can do this for all arcs of $E$. Let $D^*$ be the set of unused remaining colors in~$D$.

\paragraph{Step 4. Absorption of $B^*\cup D^*$ using the color absorber.} 

Let $S^* = B^*\cup C^*\cup D^*$. 
We now consider the collection ${\TT_{S^*}[W_{\tau}\cup \{w_{\tau},w_{\tau+1}\}]}$.
By \ref{eq: SDR1} and \eqref{eq:size of B and D star} , we have $|B^*\cup D^*|\leq r + \mu^2 n < \frac{1}{25}|S^*|$. 
Also by \ref{eq: SDR2}, at least $2r-2(r-1)\geq 2$ choices of $C(\arc{w_{\tau}w})\cap D$ and $C(\arc{w w_{\tau+1}})$ are available for each $w\in W_{\tau}$.
Hence we can apply Lemma~\ref{lem:embed-linear-fixed-color} on $\TT_{S^*}[W_{\tau}\cup \{w_{\tau},w_{\tau+1}\}]$ with $B^*\cup D^*$ playing the role of $B$ to obtain a $(S^*, B^*\cup D^*)$-rainbow path from $w_{\tau}$ to $w_{\tau+1}$.
This, together with $Q_1\cup Q_2$, provides a partially colored rainbow Hamilton path $P^{\text{final}}$ in $\TT$, which uses all the colors outside $C$. 
Let $C'\subseteq C$ be the set of colors in $C$ which are not used in $P^{\text{final}}$. By using \ref{eq: color absorbing}, we can color the remaining uncolored arcs in $Q_1$ exactly using the colors in $A\cup C'$. This provides a rainbow Hamilton path in $\TT$ starting from $w_0$ to $w_r$. 
\end{proof}

\section{Transversal Hamilton cycles}
\label{sec:rainbow-dir-ham-cycles}

The definition of strong connectedness says that for any pair of vertices, there exists a path connecting those vertices. The following lemma guarantees a similar property in the rainbow setting.

\begin{lemma} \label{lem:rainbow-strongly-connected-path}
    Let $\TT$ be a collection of strongly connected tournaments with \(|\TT|\geq |V(\TT)|-1\geq 1\). Then, for all distinct $x, y\in V(\TT)$, there exists a rainbow path from $x$ to $y$.
\end{lemma}
    
\begin{proof}
    Let \(x,y\in V(\TT)\) be two arbitrary distinct vertices.
    Let $|V(\TT)|=n$ and $\TT \supseteq \{T_1, \dots, T_{n-1}\}$. Let $L_0 \defeq \{x\}$, and inductively define for each $i\in [n-1]$
    \[
    L_i \defeq L_{i-1}\cup \{v\in V(\TT) : \text{there exists }u\in L_{i-1} \text{ such that } \arc{uv}\in E(T_i)\}.
    \]
    Then for each \(i\in [n-1]\), we have either $|L_{i-1}| < |L_i|$ or $|L_{i-1}| = n$. This is because if $|L_{i-1}| < n$ and $|L_{i-1}| = |L_i|$, it means $(V(\TT)\backslash L_i) \outdir L_i$ in the tournament $T_i$. This contradicts the assumption that $T_i$ is strongly connected. Thus, there is $r\in [n-1]$ such that $|L_r| = n$, i.e., \(L_r=V(\TT)\). Consequently, we can take a minimum $r\in [n-1]$ such that $y\in L_r$. Then, by definition of the sets $L_i$, there exists a rainbow  path from $x$ to $y$ with using some colors from~$[r]$. 
\end{proof}

Now we prove  \Cref{thm:rainbow-dir-ham-cycle}.
\begin{proof}[Proof of \Cref{thm:rainbow-dir-ham-cycle}]
    Let $\TT = \{T_1, \dots, T_n\}$ and $|V(\TT)|=n$. We may assume $T_1, \dots, T_{n-1}$ are strongly connected tournaments. 
    Assume that there are no rainbow Hamilton cycles.
    Let $\mu$ be a small constant so that we have  $0< 1/n \ll \mu \ll 1$.
    
    Fix a tournament $T\subseteq \TT^{1/2}$. There is an $\Hpart(\mu n, 1/6)$-partition $(W_0, W_1, \dots, W_{r+1}, w_0, \dots, w_{r})$ of $T$ by \Cref{lem:H-partition}. 
    Consider a Hamilton path $P_0$ of $T[W_0]$ from a vertex $x$ to a vertex $x'$ and another Hamilton path $P_{r+1}$ of $T[W_{r+1}]$ from a vertex $y'$ to a vertex $y$. We fix these vertices \(x\), \(x'\), \(y\), \(y'\). 
    
    \begin{claim}\label{eq: closing with path}
        For each $i\in [n]$, we have $\arc{xy}\in T_i$.
        Also, there are at most $10 \mu n$ internally disjoint paths of length three from $y$ to $x$ such that each path is rainbow, while different paths may have arcs with the same color.
    \end{claim}
    \begin{claimproof}
        We first define a color set $I\subseteq [n]$ and a vertex set $X$ and a path $P$ in each of the following two cases.
        \begin{Cases}
            \item If $\arc{yx}\in T_i$ for some $i\in [n]$, we let $I=\{i\}$ and $X=\emptyset$ and $P=\arc{yx}$.

            \item If Case~1 does not happen and there are more than $10\mu n$ internally disjoint rainbow paths of length three from $y$ to $x$, then we  choose one such path $P=y \rightarrow u \rightarrow v \rightarrow x$ where $u,v\notin \{ w_1,\dots, w_r\}\cup W_0\cup W_{r+1}$. Indeed, as $1/n\ll \mu$, such a path exists since $|\{ w_1,\dots, w_r\}\cup W_0\cup W_{r+1}| \le 6/\mu + 2\mu n < 10 \mu n$. In this case, we let $I$ to be the set of three colors in the rainbow path $P$, and $X= \{u,v\}$.
        \end{Cases}
        
       In either case, we let $W'_i = W_i\setminus X$ for each $i\in \{0\}\cup [r+1]$ and $\TT'= \TT_{[n]\setminus I}[V(\TT)\setminus X]$ and $T'=T\setminus X$. 
       Then, $(W'_0, W'_1, \dots,W'_r, W'_{r+1}, w_0, \dots, w_{r})$ is an $\Hpart(\mu n, 1/10)$-partition of $T'$. Note that $T'\subseteq \TT'^{1/3}$ because we assumed $T\subseteq \TT^{1/2}$. 
       Note that $W'_0=W_0$, $W'_{r+1}=W_{r+1}$, and $P_0$ and $P_{r+1}$ are still Hamilton paths on $W'_0$ and $W'_{r+1}$, respectively. Consider the path $P'_0$ obtained by appending the arc $\arc{x'w_0}$ at the end of $P_0$, and the path $P'_{r+1}$ obtained by appending the arc $\arc{w_r y'}$ in front of $P_{r+1}$.
    
       We greedily color the arcs of $P'_0$ and $P'_{r+1}$ so that $P'_0 \cup P'_{r+1}$ forms a $([n]\setminus I)$-rainbow digraph.
       This is possible as $P'_0 \cup P'_{r+1}$ contains at most $2\mu n$ arcs while each arc in $T\subseteq \TT^{1/2}$ has at least $n/2 - |I| > 2\mu n$ available colors outside of $I$. Let $C$ be the colors in $[n]\setminus I$ which are not used in the path $P'_{r+1} P P'_0$. 
       Let $V' = V(\TT')\setminus V(P_{r+1} P P_0)$.
       Then $\TT_C[V']$ is a collection of $|C|=|V'|-1$ many tournaments. As each arc of $T$ is in $T_i$ for at least \(\frac{1}{3}\Gamma(\TT') - |E(P_{r+1}'PP_0')| \ge \frac{1}{3}(n-|I|) - (2\mu n+3)\ge \frac{1}{10}|V'|\) many choices of $i\in C$, we can apply \Cref{lem:rainbowDHP} to $\TT_C[V']$ with $\gamma=\alpha=1/10$ to obtain a $C$-rainbow Hamilton path in $\TT_C[V']$ from $w_0$ to $w_r$. This together with the rainbow path $P'_{r+1} P P'_0$ yields a rainbow Hamilton cycle, a contradiction. This proves the claim.
    \end{claimproof}

\begin{claim}\label{prop: one two}
Let $y_1$ and $y_k$ be two vertices in $V(\TT)$ and $A\subseteq \Gamma(\TT)$ be a set of colors. 
Let $Q = y_1 \to y_2 \to \cdots \to y_k$ be a longest rainbow path in $\TT$ from the vertex $y_1$ to the vertex $y_k$ 
that does not use any color in $A$. If $\arc{y_\ell z} \in T_i$ for some $z\notin V(Q)$ and $i\in A$ and $\ell \in [k-1]$, then $\arc{y_{\ell+1} z}\in T_j$ for all $j\in A\setminus \{i\}$.
Symmetrically, if a vertex $z$ and a color $j\in A$ satisfies $\arc{zy_{\ell+1}} \in T_j$ for some $\ell\in [k-1]$, then $\arc{zy_{\ell}} \in T_i$ for all $i\in A\setminus \{j\}$.
\end{claim}
\begin{claimproof}
Indeed, if not, then $y_1\rarc \cdots \rarc y_{\ell} \crar{i} z \crar{j} y_{\ell+1} \rarc \cdots \rarc y_k$ yields a longer rainbow path, where the arcs other than $\arc{y_{\ell}z},\arc{zy_{\ell+1}}$ have the same color as in $Q$. This is a contradiction to the maximality of~$Q$. 
A similar argument shows the symmetric statement, and this proves the claim.
\end{claimproof}

By \Cref{lem:rainbow-strongly-connected-path}, there is a rainbow path from $y$ and to $x$ using colors in $[n-1]$. Moreover, \Cref{eq: closing with path} implies the path has length at least two. Let \(P= x_1\rarc x_2\rarc \cdots \rarc x_k\) be a longest path among those rainbow paths from \(y\) to \(x\), and let $C$ be the set of colors used in $P$.
We first show that $P$ has less than $n-1$ vertices.

\begin{claim}
We have $k< n-1$.
\end{claim}
\begin{claimproof}
If $k=n$, then let $[n]\setminus C=\{a\}$. By \Cref{eq: closing with path}, the arc $\arc{xy}$ in $T_a$ together with the $C$-rainbow path $P$ yields a rainbow Hamilton cycle, a contradiction. Thus, we can assume that $k\leq n-1$. 

Assume that $k=n-1$ and let $z$ be the unique vertex in $V(\TT)\setminus V(P)$ and let $[n]\setminus C= \{a,b\}$ be the set of colors not used in $P$.
Note that there could be multiple choices of rainbow path $P$ from $y$ to $x$ of length $n-2$. In any of such path $P$, we have the following.
\begin{equation}\label{eq: patterns}
    \begin{minipage}[c]{0.9\textwidth}
    For each $i\in [k]$ and $\{c,c'\}=\{a,b\}$, $\arc{x_iz}\in T_c$ if and only if $\arc{x_{i+1}z}\in T_{c'}$, where $x_{k+1}=x_1$.
    \end{minipage}
\end{equation}
Indeed, if $\arc{x_iz}\in T_c$ for some $i\in [k-1]$, then \Cref{prop: one two} implies that $\arc{x_{i+1}z}\in T_{c'}$. 
Moreover, if $\arc{x_{k}z}\in T_{c}$, then we have $\arc{x_1z}\in T_{c'}$ as otherwise ${x_1\rarc \dots\rarc x_k \crar{c} z \crar{c'} x_1}$ yields a rainbow Hamilton cycle. 
If $\arc{x_iz}\in T_c$ for some $i\in [k]$, then the preceding argument applied twice yields $\arc{x_{i+2}z}\in T_{c}$.
By repeatedly applying \Cref{prop: one two} $k-1$ times when $k$ is even and $2k-1$ times when $k$ is odd, we obtain that $\arc{x_{i-1}z}\in T_{c'}$.

Since one of $T_a$ and $T_b$ is strongly connected, without loss of generality, we may assume that $T_a$ is strongly connected; thus, there exists $i\in [k]$ such that $\arc{x_iz}\in T_a$ and $j\in [k]$ such that $\arc{zx_j}\in T_a$.

We claim that $k$ is even. 
If $k$ is odd, then by \eqref{eq: patterns} and the fact that $k$ is odd, we conclude that $\arc{x_{i'}z}\in T_a$ and $\arc{x_{i'} z}\in T_b$ for all $i'\in [k]$, which is a contradiction as $T_a$ is strongly connected.

Now, when $k$ is even, \eqref{eq: patterns} ensures the following.
 \begin{proplist}
    \item $\arc{x_{i'}z}\in T_a$ and $\arc{x_{i'+1}z}\in T_b$ for all $i'$ having the same parity with $i$, and \label{eq: patterns 11}
    \item $\arc{zx_{i'}}\in T_a$ and $\arc{zx_{i'+1}}\in T_b$ with all $i'$ having the same parity with $j$.  \label{eq: patterns 12}
\end{proplist}
Hence, $i$ and $j$ have different parity.

Now, for any $1<\ell<k$, assume that the arcs $\arc{x_{\ell-1}x_{\ell}}$ and $\arc{x_{\ell}x_{\ell+1}}$ are colored with $d$ and $d'$, respectively in $P$.
We can consider a rainbow path 
$P'= x_1 \rarc \cdots \rarc x_{\ell-1} \crar{c} z \crar{c'} x_{\ell+1} \rarc \cdots \rarc x_k$ where $(c,c')\in \{(a,b), (b,a)\}$ is chosen according to the parity of $\ell$.
Note that this $P'$ is also a longest rainbow path from $y$ to $x$, hence \ref{eq: patterns 11} and \ref{eq: patterns 12} also hold with the path $P'$ and the colors $d$ and $d'$. So, by swapping $d$ and $d'$ if necessary, we know that $\arc{yx_{\ell}},\arc{x_{\ell}x} \in T_{d}$. 
By choosing another $1<\ell'<k$ with $|\ell-\ell'|>1$, the same argument yields four distinct colors $d,d',d^* ,d^{**}$ such that 
$\arc{yx_{\ell'}},\arc{x_{\ell'}x} \in T_{d^*}$. 
Depending on whether $\arc{x_{\ell}x_{\ell'}}\in T_{d'}$ or $\arc{x_{\ell'}x_{\ell}}\in T_{d'}$, we have one rainbow path $yx_{\ell} x_{\ell'} x$ or $yx_{\ell'} x_{\ell} x$ of length three. We can pair up the number in $\{2,\dots, k\}$ so that paired-up numbers differ by more than $1$, then this provides at least $(n-1)/2$ internally disjoint paths of length three from $y$ to $x$ where each path is rainbow. This contradicts \Cref{eq: closing with path} and thus proves the claim that $k<n-1$. 
\end{claimproof}

Thus, we have $k<n-1$. Let $D=[n]\setminus C$. Then, we have $|D| = n-(k-1) \geq 3$.
If a vertex $z\notin V(P)$ and a color $i\in D$ satisfy $\arc{x_1z} \in T_i$, 
then \Cref{prop: one two} implies that 
$\arc{x_2z} \in T_j$ for all $j\in D\setminus \{i\}$. As $D\setminus \{i\}$ contains at least two colors, we apply \Cref{prop: one two} again for each $j\in D\setminus \{i\}$, then we conclude that $\arc{x_\ell z}\in T_i$ for all $i\in D$ and $\ell\geq 3$. Since $k\geq 3$, this implies that $\arc{x_{k} z} \in T_j$ for all $j\in D$. 
This shows that every vertex outside $V(P)$ belongs to at least one of the following two sets.
\[
    S^{+} \defeq \{ z\in V(\TT)\setminus V(P): \arc{zx_1}\in T_i \text{ for all }i\in D \}
    \quad\text{and}\quad
    S^{-} \defeq \{ z\in V(\TT)\setminus V(P): \arc{x_kz}\in T_i \text{ for all }i\in D \}.
\]
We now extend the path $P$ to a rainbow path $P_1$, which is a longest rainbow path among the rainbow paths satisfying the following. 
\begin{proplist}
    \item $P_1 = x_1\to x_2 \to \dots \to x_k \to \dots \to x_{k_1}$.
    \item $x_{k_1} \in S^+ \cup \{x_k\}$.\label{eq: xk1}
\end{proplist}

Let $D_1$ be the set of colors that do not appear in $P_1.$ Let $S_1^+ \defeq S^+\setminus V(P_1)$ and $S_1^- \defeq S^-\setminus V(P_1).$ 
As $x_{k_1}$ either belongs to $S^+$ or equals $x_k$, either the definition of $S^+$ or \Cref{eq: closing with path} implies that $\arc{x_{k_1}x_1} \in T_i$ for each $i\in D_1$.

\begin{claim}\label{cl: S+S-}
For all $z\in S^+_1$, $i\in [k_1]$, and $j\in D_1$, we have $\arc{zx_i}\in T_j$. Moreover, for all $z'\in S^-_1$  and $j\in D_1$, we have $\arc{x_{k_1} z'} \in T_j$. 
\end{claim}
\begin{claimproof}
Note that if $S^+_1$ or $S^-_1$ is empty, the corresponding statement vacuously holds. If at least one of them is not empty, we have $|D_1|\geq 2$. 

We first prove the first statement. If $S^+_1$ is not empty, choose a vertex $z\in S^+_1$. 
For this choice of $z$, consider the largest $i\in [k_1]$, if exists, such that there exists $j\in D_1$ with $\arc{x_{i}z} \in T_j$. If $i=k_1$, then $x_1\to \dots \to x_{k_1} \crar{j} z$ yields a longer rainbow path than $P_1$, contradicting the maximality of~$P_1$. Otherwise, choose $j'\in D_1\setminus \{j\}$. If $i\geq k$, then the path $x_1\to \dots \to x_{i}\crar{j} z\crar{j'} x_{i+1}\to \dots \to x_{k_1}$ yields a longer rainbow path than $P_1$, a contradiction to the maximality of $P_1$. If $i<k$, then the path $x_1\to \dots \to x_{i}\crar{j} z\crar{j'} x_{i+1}\to \dots \to x_k$ contradicts the maximality of $P$. 
Hence $i$ does not exist and this proves the first part of the claim.

Similarly, choose $z'\in S_1^-$.
Consider the largest $i\geq 0$, if exists, such that $\arc{z' x_{k+i}} \in T_j$ for some $j\in D_1$. By the definition of $S^-$, the number $i$ is positive if exists. 
Then for $j'\in D_1\setminus\{j\}$, the rainbow path $x_1\to \dots \to x_{k+i-1}\crar{j'} z'\crar{j} x_{k+i}\to \dots \to x_{k_1}$ contradicts the maximality of $P_1$, a contradiction. Hence $i$ does not exist and this proves the moreover part of the claim.
\end{claimproof}

    \begin{claim}\label{clm: S1+ empty}
        We have $S^+_1 = \emptyset$.
    \end{claim}
    \begin{claimproof}
        We first claim that $S^+_1 \outdir S^-_1$ in $T_j$ for all $j\in D_1$.
        If not, then there exist vertices $z_1\in S_1^+$ and $z_2\in S_1^-$, and $j\in D_1$ such that $\arc{z_2z_1}\in T_j$. Existence of such vertices $z_1$, $z_2$ implies $|D_1|\geq 2$, so we can choose $j'\in D_1\setminus\{j\}$. Then \Cref{cl: S+S-} implies that we can extend $P_1$ to obtain a longer rainbow path $x_1\rarc \cdots \rarc x_{k_1} \crar{j'} z_2 \crar{j} z_1$, contradicting the maximality of $P_1$. Hence, $S^+_1 \outdir S^-_1$ holds in $T_j$ for all $j\in D_1$.
    Then \Cref{cl: S+S-} implies that  $S^+_1 \outdir V(\TT)\setminus S^+_1$ holds for all $j\in D_1$, ensuring that $T_j$ is not strongly connected if $S^+_1$ is not empty. This yields at least $|D_1|\geq 2$ not strongly connected tournaments in $\TT$, a contradiction. Thus $S^+_1$ must be empty. This proves the claim.
    \end{claimproof}
    
We now extend the path $P_1$ to a rainbow path $P_2$, which is a longest rainbow path among the rainbow paths satisfying the following.
 \begin{proplist}
        \item $P_2 = y_1 \to y_2 \to \dots \to y_{\ell} \to x_1\to x_2 \to \dots \to x_{k_1}$.

        \item $y_1 \in S^-_1 \cup \{x_1\}$.\label{eq: y1-}
    \end{proplist}
Let $D_2$ be the set of colors that do not appear in the rainbow path $P_2.$ Let $S^-_2 \defeq S^-_1 \setminus V(P_2).$ 
For each $i\in [k_1]$, let $y_{\ell+i}=x_i$, then we have $P_2= y_1\to \dots \to y_{k_2}$ where $k_2=\ell+k_1$.
Then the following claim holds.

\begin{claim}\label{cl: S-2}
    For each $j\in D_2$, we have $V(\TT)\setminus S^-_2 \outdir S^-_2$ in $T_j$.
\end{claim}
\begin{claimproof}
    This is vacuously true if $S^-_2 =\emptyset$.
    Otherwise, we have $|D_1|\geq 2$ and we choose $z\in S^-_2$. 
    For this choice, consider smallest $i\in [k_2]$ such that there exists $j\in D_2$ with $\arc{zy_{i}} \in T_j$. If $i=1$, then $z\to y_1\to \dots \to y_{k_2}$ yields a longer rainbow path than $P_2$, contradicting the maximality of $P_2$. Otherwise, for $j'\in D_2\setminus \{j\}$, the path $y_1\to \dots \to y_{i-1}\crar{j'} z\crar{j} y_{i}\to \dots \to y_{k_2}$ yields a longer rainbow path than $P_2$, a contradiction. This proves the claim.
\end{claimproof}

If $S^-_2$ is not empty, then the above claim states that $T_j$ is not strongly connected for all $j\in D_2$. Then we obtain at least $|D_2|\geq 2$ tournaments in $\TT$ which are not strongly connected, a contradiction.
Hence we have $ S^-_2=\emptyset$ and $V(P_2) = V(\TT)$ and $|D_2| = 1$. Let $c$ be the only color in $D_2$. 

The moreover part of \Cref{cl: S+S-} implies that $\arc{x_{k_1}y_1}\in T_c$ if $y_1\in S^-_1$, which together with \(P_2\) gives a rainbow Hamilton cycle. 

Assume $y_1\notin S^-_1$, then \ref{eq: y1-} yields $y_1 = x_1$. In this case, \ref{eq: xk1} implies that $x_{k_1}\in S^+\cup \{x_k\}$. 
If $x_{k_1}\in S^+$, then the definition of $S^+$ yields $\arc{x_{k_1}y_1}\in T_c$ and if $x_{k_1}=x_k$, then the first part of \Cref{eq: closing with path} yields $\arc{x_{k_1}y_1}\in T_c$. In either case, this arc $\arc{x_{k_1}y_1}\in T_c$ together with $P_2$ yields a desired rainbow Hamilton cycle. This finishes the proof of \Cref{thm:rainbow-dir-ham-cycle}.
\end{proof}


\section{Concluding remarks}\label{sec:concluding}

As the transversal in a collection of tournaments has not been considered before, many interesting questions can be studied. 
Like we have considered consistently oriented Hamilton cycles, one may consider Hamilton cycles with other possible orientations. Indeed, a non-transversal version of such a question was raised by Rosenfeld~\cite{rosenfeld1972antidirected} that every sufficiently large tournament contains Hamilton cycles with all possible orientations possibly except the consistent one. This was later confirmed by Thomason~\cite{thomason1986paths}.
In the forthcoming paper~\cite{chakraborti2024hamilton}, we proved transversal generalizations of this result, both for Hamilton paths and Hamilton cycles with all possible orientations. 

Consider the following natural generalization of strong connectivity: 
a collection \(\TT\) of tournaments is \emph{strongly rainbow-connected} if for every $x,y\in V(\TT)$, there is a rainbow directed path from $x$ to $y$. Observe that if \(\TT\) contains a rainbow Hamilton cycle, then it is strongly rainbow-connected. However, unlike in the non-transversal setting, the converse is not always true by the example of $n=3$ in the introduction. 
Observe that by \Cref{lem:rainbow-strongly-connected-path}, the assumptions of \Cref{thm:rainbow-dir-ham-cycle} imply that $\TT$ is strongly rainbow-connected. We conjecture that this weaker assumption (instead of requiring all tournaments but one are strongly connected) is already sufficient to obtain a transversal Hamilton cycle in \Cref{thm:rainbow-dir-ham-cycle} if $n$ is sufficiently large. 

Another natural question is regarding pancyclicity.
Under the same assumption with \Cref{thm:rainbow-dir-ham-cycle}, does the collection $\TT$ contain rainbow directed cycles of all lengths between $3$ and $n$? We believe that our method in this paper will be useful for proving this.
Also, it is natural to ask to pin down the smallest number $n$ such that our results hold. Both \Cref{thm:rainbow-dir-ham-path,thm:rainbow-dir-ham-cycle} require \(n\) to be sufficiently large. On the other hand, as we discussed after the statements of \Cref{thm:rainbow-dir-ham-path,thm:rainbow-dir-ham-cycle}, those theorems fail to be true for $n=3$. Determining the smallest $n$ to make the statement to be true would be interesting.

Instead of tournaments, one may also consider a collection of digraphs.
Ghouila-Houri~\cite{ghouila-houri1960condition} proved the digraph version of Dirac's theorem stating that every $n$-vertex digraph with minimum semi-degree at least $n/2$ contains a directed Hamilton cycle. Here, minimum semi-degree is the minimum of out-degrees and in-degrees over all vertices. Considering a transversal version of this theorem would be interesting. It is not difficult to show the existence of rainbow Hamilton cycle if one additional digraph is given, i.e. if there are $n+1$ digraphs on $n$ vertices satisfying the semi-degree conditions. Also, the proof of \cite{cheng2021pancyclicity} can be easily adapted to show that the semi-degree bound $(1/2+o(1))n$ is sufficient to obtain a rainbow Hamilton cycle.


\subsection*{Acknowledgement}
The authors would like to thank anonymous referees for their detailed comments and helpful advice to improve the presentation of this paper.


\printbibliography

\end{document}